\documentclass[11pt]{article}
\usepackage{amsmath, amssymb, amsfonts, amsthm}
\usepackage{geometry}

\newtheorem{theorem}{Theorem}
\newtheorem{lemma}{Lemma}

\newtheorem{proposition}{Proposition}

\theoremstyle{definition}

\theoremstyle{remark}

\begin{document}

\begin{center}
{\small \rule{3cm}{0.5pt}}\\[0pt]
\textbf{Hyperbolic summation involving certain arithmetic functions}

\textbf{and the integer part function}\bigskip

\textsf{Meselem KARRAS}$^{\left( 1\right) }$\textsf{\ and Mihoub BOUDERBALA}$%
^{\left( 2\right) }$

\smallskip

$^{\left( 1\right) }$\textrm{Tissemsilt University}

\textrm{Faculty of Science and Technology}

\textrm{FIMA Laboratory, Khemis Miliana, Algeria.}

\textrm{e-mail: m.karras@univ-tissemsilt.dz\bigskip }

\textrm{\ }$^{\left( 2\right) }$\textrm{Khemis Miliana University ( UDBKM)}

\textrm{FIMA Laboratory, Faculty of matter sciences and Computer Sciences \\[%
0pt]
Rue Thniet El Had, Khemis Miliana, 44225, Ain Defla Province, Algeria\\[0pt]
e-mail:}\texttt{~}\textrm{mihoub75bouder@gmail.com}

\bigskip \smallskip

\bigskip
\end{center}

\textbf{Abstract.} Let $f$ be an arithmetic function satisfying certain
conditions. In this paper, we give an asymptotic formula for the sum%
\begin{equation*}
\sum_{n_{1}n_{2}\cdots n_{r}\leq x}f\left( \left\lfloor \frac{x}{%
n_{1}n_{2}\cdots n_{r}}\right\rfloor \right) \text{ }\left( r\geq 2\right) ,
\end{equation*}%
where $\left\lfloor .\right\rfloor $ denotes the integer part function.

\bigskip



\section{\textbf{Introduction}}

\noindent Let $\lfloor t\rfloor $ denote the integer part of the number $t$
and let $f$ be an arithmetic function that verifies some conditions. Define
the following sum%
\begin{equation}
S_{f}(x):=\sum_{n\leq x}f\left( \left\lfloor \frac{x}{n}\right\rfloor
\right) .  \label{2}
\end{equation}%
Many researchers have been interested in the sum $\left( \ref{2}\right) $.
For more information, see, for example, Bordell\`{e}s, and others \cite%
{BDHPS}. Independently, in \cite{W}, \cite{Z} the others established that
for an arithmetic function $f(n)$\ satisfying $f(n)\ll n^{\alpha }\left(
\log {n}\right) ^{\theta }$, where $\alpha \in \lbrack 0,1)$ and $\theta
\geq 0,$\ the sum $\left( \ref{2}\right) $ can be expressed as%
\begin{equation*}
S_{f}(x)=x\sum_{n=1}^{\infty }\frac{f(n)}{n(n+1)}+\mathcal{O}\left(
x^{(1+\alpha )/2}\left( \log {x}\right) ^{\theta }\right) .
\end{equation*}%
For $f=\tau $, Ma and Sun \cite{MS} proved the refined result%
\begin{equation*}
S_{\tau }(x)=\sum_{n\leq x}\tau \left( \left\lfloor \frac{x}{n}\right\rfloor
\right) =x\sum_{n=1}^{\infty }\frac{\tau (n)}{n(n+1)}+\mathcal{O}\left(
x^{11/23+\varepsilon }\right) .
\end{equation*}%
Recent research further refined the error terms for multiplicative
functions. In contrast, additive functions, such as $\omega (n)$, the number
of distinct prime divisors, present unique challenges. For instance, in \cite%
{karras and Bouder} it was shown that 
\begin{equation*}
\sum_{n\leq x}\omega \left( \left\lfloor \frac{x}{n}\right\rfloor \right)
=Cx+\mathcal{O}\left( x^{1/2}\log x\right) ,
\end{equation*}%
where $C\approx 0.5918$ is a constant defined by 
\begin{equation*}
C=\sum\limits_{n=1}^{\infty }\dfrac{\omega \left( n\right) }{n\left(
n+1\right) }.
\end{equation*}%
Bordell\`{e}s \cite{Bordell} has improved the error term and shown that 
\begin{equation*}
\sum_{n\leq x}\omega \left( \left\lfloor \frac{x}{n}\right\rfloor \right)
=Cx+\mathcal{O}\left( x^{455/914}\log x\right) ,
\end{equation*}%
where $455/914\approx 0.4978.$

\noindent In the general case, we can study sums of the form%
\begin{equation*}
S_{f,r}(x)=\sum_{n_{1}n_{2}\cdots n_{r}\leq x}f\left( \left\lfloor \frac{x}{%
n_{1}n_{2}\cdots n_{r}}\right\rfloor \right) ,
\end{equation*}%
where $r\geq 2$. This can be expressed as%
\begin{equation*}
S_{f,r}(x)=\sum_{n\leq x}f\left( \left\lfloor \frac{x}{n}\right\rfloor
\right) \tau _{r}\left( n\right) ,
\end{equation*}%
where the function $\tau _{r}\left( n\right) $ counts the ways to write $n$
as a product of $r$ positive integers with $\tau _{2}\left( n\right) =\tau
\left( n\right) $. In particular, in \cite{KA LI ST} for $r=2,$ the authors
established the following result%
\begin{equation*}
\sum_{n_{1}n_{2}\leq x}f\left( \left\lfloor \frac{x}{n_{1}n_{2}}%
\right\rfloor \right) =C_{1}(f)x\log {x}+C_{2}(f)x+\mathcal{O}\left( x^{%
\frac{4+3\alpha }{7}+\varepsilon }\right) ,
\end{equation*}%
where $\gamma $ is the Euler-Mascheroni constant and $C_{1}(f),$ $C_{2}(f)$
are two constants.

We are now ready to present our main results.

\section{\textbf{\ Main results}}

\begin{theorem}
Let $r\geq 2$ be an integer and let $f$ be an arithmetic function and assume
that there exists $\alpha \in \left[ 0,1\right) $ such that $\ $%
\begin{equation*}
f(n)\ll n^{\alpha },
\end{equation*}%
and%
\begin{equation*}
\sum_{n\leq x}f(n)\ll x\left( \log \log x\right) ^{\beta },\text{ }\beta
\geq 0.
\end{equation*}%
Then%
\begin{equation*}
S_{f,r}(x)=x\sum_{j=0}^{r-1}\sum_{i=0}^{j}\left( 
\begin{array}{c}
j \\ 
i%
\end{array}%
\right) \left( -1\right) ^{i}a_{j}C_{i}\left( \log x\right) ^{j-i}+\mathcal{O%
}\left( x\left( \log x\right) ^{r-2}\left( \log \log x\right) ^{\beta
}\right) ,
\end{equation*}%
such that $C_{i}$ and $a_{j}$ $($where $0\leq i\leq j\leq r-1)$ are real
constants.
\end{theorem}

\begin{lemma}
Let $f$ be an arithmetic function, and suppose that there exists a constant $%
\alpha \in \left[ 0,1\right) $ such that$\ $%
\begin{equation*}
f(n)\ll n^{\alpha }.
\end{equation*}%
Then the series 
\begin{equation*}
\sum_{n=1}^{\infty }f\left( n\right) \left( \frac{\left( \log n\right) ^{i}}{%
n}-\frac{\log \left( n+1\right) ^{i}}{n+1}\right)
\end{equation*}%
is convergent for any fixed integer $i\geq 0.$
\end{lemma}

\begin{proof}
As the integer $n$ tends to infinity, we have the following approximation%
\begin{eqnarray*}
\log \left( n+1\right) &=&\log n+\log (1+\frac{1}{n}) \\
&=&\log n+\frac{1}{n}-\frac{1}{2n^{2}}+\mathcal{O}\left( \frac{1}{n^{3}}%
\right) \\
&=&\log n\left( 1+\frac{1}{n\log n}-\frac{1}{2n^{2}\log n}+\mathcal{O}\left( 
\frac{1}{n^{3}\log n}\right) \right) .
\end{eqnarray*}%
Thus, for any fixed integer $i\geq 0$,%
\begin{equation*}
\left( \log \left( n+1\right) \right) ^{i}=\left( \log n\right) ^{i}\left( 1+%
\frac{1}{n\log n}-\frac{1}{2n^{2}\log n}+\mathcal{O}\left( \frac{1}{%
n^{3}\log n}\right) \right) ^{i}.
\end{equation*}%
Applying the binomial expansion,%
\begin{equation*}
\left( \log \left( n+1\right) \right) ^{i}=\left( \log n\right) ^{i}+i\frac{%
\left( \log n\right) ^{i-1}}{n}+\mathcal{O}\left( \frac{\left( \log n\right)
^{i-2}}{n^{2}}\right) .
\end{equation*}%
Hence, the term 
\begin{equation*}
f\left( n\right) \left( \frac{\left( \log n\right) ^{i}}{n}-\frac{\log
\left( n+1\right) ^{i}}{n+1}\right) \ll f\left( n\right) \frac{\left( \log
n\right) ^{i}}{n^{2}}\ll \frac{\left( \log n\right) ^{i}}{n^{2-\alpha }}.
\end{equation*}%
For $\alpha \in \left[ 0,1\right) ,$ the series 
\begin{equation*}
\sum\limits_{n=1}^{\infty }\dfrac{\left( \log n\right) ^{i}}{n^{2-\alpha }}
\end{equation*}%
is convergent for any integer $i\geq 0.$
\end{proof}

\begin{proof}[Proof of Theorem 1]
Firstly, we have 
\begin{equation*}
S_{f,r}(x)=\sum_{n\leq x}f\left( \left\lfloor \frac{x}{n}\right\rfloor
\right) \tau _{r}\left( n\right) .
\end{equation*}%
By setting $d=\left\lfloor \frac{x}{n}\right\rfloor ,$ we obtain 
\begin{eqnarray*}
S_{f,r}(x) &=&\sum_{n\leq x}\sum_{d=\left\lfloor \frac{x}{n}\right\rfloor
}f\left( d\right) \tau _{r}\left( n\right) \\
&=&\sum_{d\leq x}f\left( d\right) \sum_{\frac{x}{d+1}<n\leq \frac{x}{d}}\tau
_{r}\left( n\right) .
\end{eqnarray*}%
To proceed further, we use the identity%
\begin{equation*}
\sum_{\frac{x}{d+1}<n\leq \frac{x}{d}}\tau _{r}\left( n\right) =\sum_{n\leq 
\frac{x}{d}}\tau _{r}\left( n\right) -\sum_{n\leq \frac{x}{d+1}}\tau
_{r}\left( n\right) .
\end{equation*}%
Additionally, using the known formula ( see \cite{TIMARCH}) 
\begin{equation}
\sum_{n\leq x}\tau _{r}\left( n\right) =xP_{k}\left( \log x\right) +\mathcal{%
O}\left( x^{1-\frac{1}{r}}\left( \log x\right) ^{r-2}\right) ,  \label{3}
\end{equation}%
where 
\begin{equation*}
P_{k}\left( x\right) =\sum_{i=0}^{r-1}a_{i}x^{i},
\end{equation*}%
is a polynomial of degree $r-1.$ Thus, we obtain 
\begin{eqnarray*}
S_{f,r}(x) &=&x\sum_{j=0}^{r-1}\sum_{i=0}^{j}\left( 
\begin{array}{c}
j \\ 
i%
\end{array}%
\right) \left( -1\right) ^{i}a_{j}\left( \log x\right) ^{j-i}\times
\sum_{d\leq x}f\left( d\right) \left( \frac{\left( \log d\right) ^{i}}{d}-%
\frac{\left( \log \left( d+1\right) \right) ^{i}}{d+1}\right) \\
&& \\
&&+\mathcal{O}\left( \sum_{d\leq x}f\left( d\right) \left( \frac{x}{d}%
\right) ^{1-\frac{1}{r}}\left( \log \frac{x}{d}\right) ^{r-2}\right) .
\end{eqnarray*}%
Furthermore, we use the estimate%
\begin{equation*}
\sum_{d\leq x}f\left( d\right) \left( \frac{\left( \log d\right) ^{i}}{d}-%
\frac{\log \left( d+1\right) ^{i}}{d+1}\right) =C_{i}+\mathcal{O}\left(
x^{\alpha -1}\right) ,
\end{equation*}%
where $C_{i}$\ represents the convergence constant of the infinite series 
\begin{equation*}
\sum_{d=1}^{\infty }f\left( d\right) \left( \frac{\left( \log d\right) ^{i}}{%
d}-\frac{\log \left( d+1\right) ^{i}}{d+1}\right) ,
\end{equation*}%
which is ensured by Lemma 1. Then we obtain%
\begin{eqnarray*}
S_{f,r}(x) &=&x\sum_{j=0}^{r-1}\sum_{i=0}^{j}\left( 
\begin{array}{c}
j \\ 
i%
\end{array}%
\right) \left( -1\right) ^{i}a_{j}\left( \log x\right) ^{j-i}\left( C_{i}+%
\mathcal{O}\left( x^{\alpha -1}\right) \right) \\
&&+\mathcal{O}\left( \sum_{d\leq x}f\left( d\right) \left( \frac{x}{d}%
\right) ^{1-\frac{1}{r}}\left( \log \frac{x}{d}\right) ^{r-2}\right) .
\end{eqnarray*}%
On the other hand, we observe that%
\begin{equation*}
\sum_{d\leq x}f\left( d\right) \left( \frac{x}{d}\right) ^{1-\frac{1}{r}%
}\left( \log \frac{x}{d}\right) ^{r-2}\ll x^{1-\frac{1}{r}}\left( \log
x\right) ^{r-2}\sum_{d\leq x}\frac{f\left( d\right) }{d^{1-\frac{1}{r}}}.
\end{equation*}%
Here, we aim to estimate the sum 
\begin{equation*}
\underset{d\leq x}{\sum }\frac{f\left( d\right) }{d^{1-\frac{1}{r}}},
\end{equation*}%
for which we apply Abel's summation formula. Thus, we can obtain the
following%
\begin{equation*}
\sum_{d\leq x}f\left( d\right) \left( \frac{x}{d}\right) ^{1-\frac{1}{r}%
}\left( \log \frac{x}{d}\right) ^{r-2}\ll x\left( \log x\right) ^{r-2}\left(
\log \log x\right) ^{\beta }.
\end{equation*}%
Finally, we conclude that 
\begin{equation*}
S_{f,r}(x)=x\sum_{j=0}^{r-1}\sum_{i=0}^{j}\left( 
\begin{array}{c}
j \\ 
i%
\end{array}%
\right) \left( -1\right) ^{i}a_{j}C_{i}\left( \log x\right) ^{j-i}+\mathcal{O%
}\left( x\left( \log x\right) ^{r-2}\left( \log \log x\right) ^{\beta
}\right) .
\end{equation*}
\end{proof}

\begin{proposition}
Let $f$ be an arithmetic function defined by 
\begin{equation*}
f\left( n\right) =\sum_{p^{\alpha _{i}}\parallel n}g\left( \alpha
_{i}\right) ,
\end{equation*}%
such that $g(n)$ is an arithmetic function verifying $g\left( 0\right) =0$, $%
g\left( 1\right) \neq 0$ and $g\left( n\right) =\mathcal{O}\left(
2^{n/2}\right) $. Keeping the same notation as in Theorem 1, we have 
\begin{equation*}
S_{f,r}(x)=x\sum_{j=0}^{r-1}\sum_{i=0}^{j}\left( 
\begin{array}{c}
j \\ 
i%
\end{array}%
\right) \left( -1\right) ^{i}a_{j}C_{i}\left( \log x\right) ^{j-i}+\mathcal{O%
}\left( x\left( \log x\right) ^{r-2}\log \log x\right) .
\end{equation*}
\end{proposition}

\begin{proof}
The function $f$ satisfies the first condition of Theorem 1. In fact, we
know that for any integer $n=p_{1}^{\alpha _{1}}p_{2}^{\alpha
_{2}}...p_{r}^{\alpha _{r}},$ we have $\alpha _{i}\leq \log n/\log 2.$
Therefore, we obtain 
\begin{equation*}
\sum_{p^{\alpha }\parallel n}g\left( \alpha \right) \ll \sum_{p^{\alpha
}\parallel n}2^{\alpha /2}\ll n^{1/2}\sum_{p^{\alpha }\parallel n}1\ll \frac{%
\log n}{\log \log n}n^{1/2}=\mathcal{O}\left( n^{\frac{1}{2}+\varepsilon
}\right) ,\text{ }0<\varepsilon <\frac{1}{2}.
\end{equation*}%
The second condition is also satisfied, with $\beta =1$ according to theorem
1 of $\cite{R. L. Duncan}$.
\end{proof}

\section{\textbf{Examples }}

\begin{enumerate}
\item For a fixed integer $k\geq 0,$ let $g\left( n\right) =n^{k},$ then we
have%
\begin{equation*}
f\left( n\right) =\Omega _{k}\left( n\right) =\sum_{p^{\alpha }\parallel
n}\alpha ^{k}.
\end{equation*}%
Notably, $\Omega _{0}\left( n\right) =\omega \left( n\right) $ and $\Omega
_{1}\left( n\right) =\Omega \left( n\right) $. Using results from $\cite%
{DUNCAN}$ or $\cite{M HASSANI}$, we find 
\begin{equation*}
\sum_{n\leq x}\Omega _{k}\left( n\right) \ll x\log \log x.
\end{equation*}%
Thus 
\begin{equation*}
S_{\Omega _{k},r}\left( x\right) =x\sum_{j=0}^{r-1}\sum_{i=0}^{j}\left( 
\begin{array}{c}
j \\ 
i%
\end{array}%
\right) \left( -1\right) ^{i}a_{j}C_{i}\left( \log x\right) ^{j-i}+\mathcal{O%
}\left( x\left( \log x\right) ^{r-2}\log \log x\right) .
\end{equation*}

\item The same formula applies by replacing $\Omega _{k}\left( n\right) $
with $\tau \left( \tau \left( n\right) \right) $ or $\Omega \left( \tau
\left( n\right) \right) $ (see \cite{Hppner}). Similarly, by replacing with $%
\dfrac{\omega \left( n\right) }{p\left( n\right) }$ or $\dfrac{\Omega \left(
n\right) }{p\left( n\right) }$, where $p\left( n\right) $ is the smallest
prime divisor of $n$ ( see\cite{ZHANG})\textbf{.}

\item For the function $f\left( n\right) =\omega ^{2}\left( n\right) ,$ (see 
\cite{Shapiro}) we have%
\begin{equation*}
S_{\omega ^{2},r}\left( x\right) =x\sum_{j=0}^{r-1}\sum_{i=0}^{j}\left( 
\begin{array}{c}
j \\ 
i%
\end{array}%
\right) \left( -1\right) ^{i}a_{j}C_{i}\left( \log x\right) ^{j-i}+O\left(
x\left( \log x\right) ^{r-2}\left( \log \log x\right) ^{2}\right) .
\end{equation*}
\end{enumerate}

\end{document}